\documentclass[12pt,centertags,oneside]{amsart}

\usepackage{amstext,amsthm,amscd}
\usepackage{amsmath,amssymb}

\usepackage[a4paper,width=16.2cm,top=3cm,bottom=3cm]{geometry}

\usepackage[mathscr]{eucal}
\usepackage{mathrsfs}
\usepackage{charter}
\usepackage{colonequals}

\usepackage{xcolor}
\usepackage[colorlinks=true,
            linkcolor=blue,
            citecolor=blue,
            urlcolor=blue]{hyperref}

\numberwithin{equation}{section}

\newtheorem{theorem}{Theorem}[section]
\newtheorem{definition}[theorem]{Definition}
\newtheorem{proposition}[theorem]{Proposition}
\newtheorem{corollary}[theorem]{Corollary}
\newtheorem{lemma}[theorem]{Lemma}
\newtheorem{remark}[theorem]{Remark}

\newtheorem{conjecture}[theorem]{Conjecture}

\newcommand{\volume}{{\rm vol}}

\newcommand{\supp}{{\rm supp}}

\newcommand{\ddc}{dd^c}

\newcommand{\PSH}{{\rm PSH}}

\newcommand{\capa}{\mathop{\mathrm{Cap}}\nolimits}

\newcommand{\D}{\mathbb{D}}
\newcommand{\C}{\mathbb{C}}

\newcommand{\N}{\mathbb{N}}
\newcommand{\Z}{\mathbb{Z}}
\newcommand{\R}{\mathbb{R}}

\renewcommand\P{\mathbb{P}}

\title[]{\quad Singularity of non-pluripolar cohomology classes}

\author{Duc-Bao Nguyen, Shuang Su, Duc-Viet Vu}

\newcommand{\Addresses}{
{
		\bigskip
  \footnotesize
		
  \textsc{Duc-Bao Nguyen, National University of Singapore, Department of Mathematics, 10 Lower Kent Ridge Road, 119076, Singapore.}
		\noindent
		\par\nopagebreak
		\noindent
		\textit{E-mail address}: \texttt{ducbao.nguyen@u.nus.edu}}
{\bigskip
\footnotesize

		\textsc{Shuang Su, Center for Complex Geometry, Institute for Basic Science, 55 Expo-ro, Yuseong-gu, Daejeon 34126, Republic of Korea}
		\noindent
		\par\nopagebreak
		\noindent
		\textit{E-mail address}: \texttt{su30@ibs.re.kr}

{
		\bigskip
		\footnotesize
		\textsc{Duc-Viet Vu, University of Cologne, Division of Mathematics, Department of Mathematics and Computer Science, Weyertal 86-90, 50931, K\"oln.}
		\noindent
		\par\nopagebreak
		\noindent
		\textit{E-mail address}: \texttt{dvu@uni-koeln.de}
}}
}
\date{\today}

\begin{document}

\hyphenpenalty=10000

\sloppy

\begin{abstract} We establish a relation between Lelong numbers and the full mass property of relative non-pluripolar products. We use this relation to prove that if the restricted volume of a big class $\alpha$ along an effective divisor $D$ is of full mass, then the Lelong numbers of the non-pluripolar class $\langle \alpha^{n-1}\rangle$ at every point in the support of $D$ are zero. In particular, we obtain that on projective manifolds, the Lelong numbers of the non-pluripolar class $\langle \alpha^{n-1}\rangle$ of a big class $\alpha$ are zero.
 
\end{abstract}


\maketitle




\section{Introduction}

A major problem in the study of the pseudoeffective $(1,1)$-cone in a compact \break K\"ahler manifold of dimension $n$ is the Boucksom-Demailly-P\u{a}un-Peternell conjecture (\cite{BDPP}), predicting that the duality of the pseudoeffective $(1,1)$-cone is equal to the cone of non-pluripolar cohomology $(n-1,n-1)$-classes. This conjecture is known to follow from a conjectural formula for the differentiability of the volume functions on the big cone. It was verified in the projective case by Witt Nystr\"om \cite{WittNystrom-duality}, while the K\"ahler case is still open. In view of recent developments \cite{WittNystrom-deform,Vu_derivative} on this problem, one is naturally led to ask whether the restricted volume of a big class to a divisor is of ``full mass". If so, it will establish the above-mentioned conjecture for divisorial partial derivatives. This question is in turn part of a more fundamental problem of understanding possible singularity patterns of the intersection of closed positive currents.  Let us now describe the situation in more detail.

Let $X$ be a compact K\"ahler manifold of dimension $n \geq 2$. For $0 \le k \le n$, recall that  a cohomology $(k,k)$-class $\alpha \in H^{k,k}(X,\R)$ is said to be pseudoeffective if it contains a closed positive $(k,k)$-current. In this case, we write $\alpha\ge 0$. For $\alpha_1, \alpha_2 \in H^{k,k}(X,\R)$, we write $\alpha_1 \ge \alpha_2$ if $\alpha_1 - \alpha_2 \ge 0$. A class $\alpha\in H^{k,k}(X,\R)$ is called \emph{non-pluripolar} if it is represented by a non-pluripolar product $\langle T_1 \wedge \cdots \wedge T_k\rangle$, where $T_1, \ldots, T_k$ are closed positive $(1,1)$-currents. We refer to \cite{BEGZ} for basics on non-pluripolar products.  For every closed positive current $R$, we denote by $\{R\}$ the cohomology class of $R$. If $V$ is an analytic set of pure dimension, we denote by $[V]$ the current of integration along $V$ and by $\{V\}$ the cohomology class of $[V]$.

Let $\alpha \in H^{1,1}(X,\R)$ be a big cohomology class. Let $D$ be an effective real divisor.  Let $\langle \alpha^{n-1}\rangle|_{X|D}$ denote the numerical restricted volume of $\alpha$ to $D$ (see \cite[Section 2]{Collins-Tosatti-nullloci}).  By \cite[Lemma 4.5 and Proposition 4.6]{Vu_derivative}, we know that $\langle \alpha^{n-1}\rangle|_{X|D}$ is equal to the cohomology class of the relative non-pluripolar product $\langle T_{\min,\alpha}^{n-1} \ \dot{\wedge} \ [D] \rangle$, where $T_{\min,\alpha}$ is a current of minimal singularities in $\alpha$. We refer to Section~\ref{sec sing positive product} for a short recap of relative non-pluripolar products from \cite{Viet-generalized-nonpluri}. 
 By \cite{Vu_derivative}, we know that 
\begin{align}\label{eq-daohamrieng}
\frac{d}{dt}\bigg|_{t=0}\volume(\alpha+ t \{D\})=n \langle \alpha^{n-1}\rangle|_{X|D},
\end{align}
The equality (\ref{eq-daohamrieng}) was proved in  \cite[Theorem C]{WittNystrom-deform} in the case where $D$ is smooth and is not  contained in the non-K\"ahler locus of $\alpha$. 
The equality (\ref{eq-daohamrieng}) is a special case of a conjecture due to Boucksom-Demailly-P\u{a}un-Peternell \cite{BDPP} that for any $\gamma \in H^{1,1}(X,\R)$ there holds
\begin{align}\label{eq-BDPP}
\frac{d}{dt}\bigg|_{t=0}\volume(\alpha+ t \gamma)=n \langle \alpha^{n-1}\rangle \smile \gamma.
\end{align}
 The case where $\alpha, \gamma$ are in the real Neron-Severi space $\mathrm{NS}_{\R}(X)$ was proved in \cite{Boucksom-derivative-volume} and \cite{BDPP}. If $X$ is projective, then (\ref{eq-BDPP}) was established by Witt Nystr\"om \cite{WittNystrom-duality}. As mentioned above, in this case, the conjecture~\eqref{eq-BDPP} is equivalent to the orthogonal property and implies other conjectural properties of the big cone such as the weak transcendental Morse inequality and the duality of the pseudoeffective cone and the movable cone. We refer to the appendix in \cite{WittNystrom-duality} for a proof of these relations. There are numerous works around this topic, for example, see \cite{Xiao-weak-morse,Popovici2,Tosatti-weakMorse,Popovici,Xiao-movable-inter,Tosatti-orthogonality}.  We notice that in general one has 
\begin{align}\label{ine-totaldensityDTmin}
 \langle \alpha^{n-1}\rangle \smile \{D\} \ge  \langle \alpha^{n-1}\rangle|_{X|D},
\end{align}
see \cite[Section 8]{Collins-Tosatti-nullloci}. As it will be clearer later, the left-hand side of this inequality can be interpreted as the total density class of $\langle T_{\min,\alpha}^{n-1}\rangle$ and $[D]$. Hence (\ref{ine-totaldensityDTmin}) tells us that the class of the relative non-pluripolar product $\langle T_{\min,\alpha}^{n-1}\ \dot{\wedge}\ [D] \rangle$ is dominated by the total density class of $\langle T_{\min,\alpha}^{n-1}\rangle$ and $[D]$. If one can prove that (\ref{ine-totaldensityDTmin}) is indeed an equality, then one obtains the conjectural formula (\ref{eq-BDPP}) in the case where $\gamma$ is the class of a divisor. This observation motivated us to study more deeply the relation between the relative non-pluripolar products and the density currents introduced in \cite{Dinh_Sibony_density}. This problem was addressed in \cite{Viet-density-nonpluripolar}.

For a closed positive current $S$ and $x\in X$, we denote by $\nu(S,x)$ the Lelong number of $S$ at $x$. Here is the first main result in our paper.

\begin{theorem}\label{theorem main compare Lelong}
Let $T_1,\ldots,T_{n-1}$ be closed positive $(1,1)$-currents on $X$. Let $T$ be a closed positive $(1,1)$-current on $X$ and $R:=\langle T_1 \wedge \cdots \wedge T_{n-1} \rangle$. Then we have 
\begin{align}\label{ine-densityrelativeRT}
\big\{\langle \wedge_{j=1}^{n-1}T_j \ \dot{\wedge} \ T\rangle \big\} \le \{R\} \smile \{T\},
\end{align}
and moreover, if equality holds, then 
\begin{align}\label{eq-fullmasslelong0}
\nu(R,x) \cdot \nu(T,x) = 0
\end{align}
 for every $x\in X$.
\end{theorem}

We will see later that $ \{R\} \smile \{T\}$ is the total density class of $R$ and $T$ (hence a pseudoeffective class), and (\ref{ine-densityrelativeRT}) follows from a comparison between density currents and relative non-pluripolar products. This comparison is somewhat similar to \cite[Theorem 3.5]{Viet-density-nonpluripolar} saying that 
$$\{\langle T \ \dot{\wedge}\ R \rangle\} \le  \{T\} \smile \{R\}=\{R\} \smile \{T\}.$$
The difference is the left-hand side of this inequality where the role of $T,R$ is somehow reversed.  We also want to point out  a subtlety that in (\ref{ine-densityrelativeRT}), we have to use density currents of $\langle \wedge_{j=1}^{n-1} T_j \rangle$ and $T$ (not those of $T_1,\ldots, T_{n-1},T$). This interpretation in terms of density currents shows us that having the equality in (\ref{ine-densityrelativeRT}) is equivalent to saying that the relative non-pluripolar product $\langle \wedge_{j=1}^{n-1} T_j \ \dot{\wedge} \ T\rangle$  is of ``full mass''. In this sense, the assertion (\ref{eq-fullmasslelong0}) reflects again the role of Lelong numbers as an obstacle to the full mass property. Such a phenomenon has already appeared in several previous works, see \cite{GZ-weighted,Lu-Darvas-DiNezza-singularitytype,Vu_lelong-bigclass}. Our approach is nevertheless different and, in a similar vein to \cite{SuVu-volume-lelong}. We underline here that Theorem \ref{theorem main compare Lelong} looks similar to \cite[Theorem 3.6]{SuVu-volume-lelong}, but the latter result only directly gives  Theorem \ref{theorem main compare Lelong} for $n=2$.

Let $\alpha \in H^{1,1}(X,\R)$ be a big class and $T_{\min,\alpha}$ be a current of minimal singularities in $\alpha$. For $1 \le k \le n$, we define $\nu(\langle \alpha^k \rangle, x)$ to be the Lelong number of  $\langle T_{\min,\alpha}^k \rangle$ at $x$. This definition is independent of the choice of $T_{\min,\alpha}$ (see Theorem~\ref{theorem Lelong of relative non-pluripolar product}). The case $k=1$ was discussed in \cite{Boucksom_anal-ENS}. The following is our second main result showing that in the case where $X$ is projective, the class $\langle \alpha^{n-1} \rangle$ bears a characteristic of a ``nef" class of higher bi-degree.

\begin{theorem}\label{cor-divisorpartialderivitve}
    Let $X$ be a projective manifold of dimension $n$. Let $\alpha \in H^{1,1}(X,\R)$ be a big class. Then, we have
    $\nu(\langle \alpha^{n-1}\rangle, x)=0$ for every $x \in X$.
\end{theorem}

This result follows from Corollary~\ref{cor-divisorpartialderivitve2}. We emphasize here a special case when $X$ is a projective surface. In this case,
Theorem~\ref{cor-divisorpartialderivitve} implies that $\langle \alpha \rangle$ is a nef class, and the equality $\alpha= \langle \alpha \rangle + \beta$ is a decomposition of $\alpha$ into the sum of a nef class and a cohomology class of an effective $\R$-divisor. It was proved in \cite{Boucksom_anal-ENS} that this is indeed the Zariski decomposition of $\alpha$ when $\alpha$ is integral. In this sense, the statement that $\nu(\langle \alpha^{n-1} \rangle,x)=0$ for $x\in X$ is related to Zariski's decompositions. Moreover, one can ask if Theorem \ref{cor-divisorpartialderivitve} remains to be true if $\langle \alpha^{n-1}\rangle$ is replaced by some $\langle \alpha^k \rangle$ for $1\le k < n-1$. This is, however, not the case in general.  By \cite[Section 5]{XiaojunWu-nefnessHigherDim} (see also \cite{Cutkosky-ZariskiDecomp,Boucksom_anal-ENS}), there is a projective three-fold $X$ and a big class $\alpha \in H^{1,1}(X,\R)$ such that $\langle \alpha \rangle$ is not nef. For such a class $\alpha$, there exists some point $x\in X$ such that $\nu(\langle \alpha\rangle,x) > 0$.

Finally, we would like to propose the following conjecture, which is a weaker version of Conjecture~\eqref{eq-BDPP} (see Theorem \ref{the-main-Zariski}).

\begin{conjecture} \label{conj-lelong} Let $X$ be a compact K\"ahler n-dimensional manifold. Let $\alpha \in H^{1,1}(X,\R)$ be a big class. Then the Lelong number of $\langle \alpha^{n-1} \rangle$ at every point in $X$ is zero. 
\end{conjecture}

\vspace{1cm}

\noindent \textbf{Acknowledgments.} The research of Duc-Bao Nguyen is supported by the Singapore International Graduate Award (SINGA) and the NUS Overseas Research Immersion Award (ORIA). He would like to thank his advisor, Tien-Cuong Dinh, for many useful discussions and for his constant support during the preparation of this work. Part of this work was carried out during Duc-Bao Nguyen's visit at the Department of Mathematics and Computer Science, University of Cologne. He would like to thank them for their warm welcome and financial support. The research of Shuang Su and Duc-Viet Vu is partially funded by the Deutsche Forschungsgemeinschaft (DFG, German Research Foundation) Projektnummer 500055552.

\section{Preliminaries on density currents}\label{sec preliminaries}

In this section, we recall basic facts on the theory of density currents introduced by Dinh-Sibony in \cite{Dinh_Sibony_density}. We refer to  \cite[Section 2]{Vu_density-nonkahler} for more details and simplified presentations of some results in \cite{Dinh_Sibony_density}.

Let $X$ be a compact K\"ahler manifold of dimension $n$. Let $V$ be a submanifold of dimension $l$. Let $T$ be a closed positive $(p,p)$-current on $X$ with $0\leq p \leq n$. We denote by $\pi: E \to V$ the normal bundle of $V$ in $X$ and $\overline{E}:= \P (E\oplus \C)$ the compactification of $E$. We also use $\pi$ to denote the projection of $E$ to $V$.
 For $\lambda \in \C^*$, we denote $A_\lambda : E\to E$ the multiplication by $\lambda$ on fibers of $E$.

Let $U$ be a local chart of $X$ such that $U\cap V \neq \varnothing$. Let $\tau$ be a smooth diffeomorphism from $U$ to an open neighborhood of $V\cap U$ in $E$ which is identity on $V\cap U$ such that the restriction of its differential $d\tau$ to $E|_{V\cap U}$ is identity. Such a map is called \textit{an admissible map}. When $U$ is a small enough local chart, we can choose a \textit{holomorphic} admissible map by using suitable holomorphic coordinates on $U$. By \cite[Theorem 4.6]{Dinh_Sibony_density}, the family of currents $(A_\lambda)_*\tau_*( T)$ has uniformly bounded mass in $\lambda$ on compact subsets in $E|_{U\cap V}$, and every limit current is independent of the choice of the admissible map $\tau$. Thus, we can obtain global limit currents as $\lambda \to \infty$. Such a limit, which is a closed positive current on $E$, is called \emph{a tangent current to $T$} along $V$.  Every tangent current to $T$ along $V$  can be trivially extended to  $\overline{E}$ (see \cite[Theorem 4.6]{Dinh_Sibony_density}) and is, in general, not unique. Its cohomology class in $H^{2p}(\overline{E},\R)$ is unique and is called \textit{total tangent class} of $T$ along $V$. We denote it by $\kappa^{V}(T)$. We note that in \cite{Dinh_Sibony_density}, the authors used global admissible maps to define tangent currents directly. The definition there is equivalent to the one we presented here by using local holomorphic maps $\tau$ as above.

\begin{definition}
Let $p_1, \ldots,p_m \in \N$ with $p:=p_1+\cdots+p_m \le n$. Let $T_j$ be closed positive $(p_j,p_j)$-currents for $1\leq j\leq m$. Denote by $\mathbb{T}$ the tensor product $T_1 \otimes \cdots \otimes T_m$ which is a closed positive $(p,p)$-current on $X^m$. A density current associated to $T_1,\ldots,T_m$ is a tangent current to $\mathbb{T}$ along the diagonal $\Delta_m:=\{(x,\ldots,x) : x\in X\} \subset X^m$ of $X^m$.   
\end{definition}

Let $\mathbb{E}_m$ be the normal bundle of $\Delta_m$ in $X^m$. Let  $\pi_m: \mathbb{E}_m \to \Delta_m \simeq X$ be the canonical projection. We recall that the density currents associated to $T_1,\ldots,T_m$ are closed positive $(p,p)$-currents on $\mathbb{E}_m$ and can be extended to currents on $\overline{\mathbb{E}}_m$. The cohomology class $\kappa(T_1,\ldots,T_m)$ of density currents of $T_1,\ldots,T_m$ is unique and called \emph{the total density class} of $T_1,\ldots, T_m$.


\begin{definition}(\cite[Definition 3.1]{Dinh_Sibony_density})
    Let $\pi_V : F \to V$ be a holomorphic submersion between complex manifolds $F$ and $V$. Let $S$ be a positive $(p,p)$-current on $F$. The horizontal dimension (or h-dimension for short) of $S$ with respect to $\pi$ is the largest number $q$ such that $S \wedge \pi_V^* (\omega_V^{q}) \neq 0$ for some Hermitian metric $\omega_V$ on $V$.
\end{definition}

It was shown in \cite[Theorem 4.6]{Dinh_Sibony_density} that the density currents of $T_1,\ldots,T_m$ have the same h-dimension which is called the \textit{density h-dimension} of $T_1,\ldots,T_m$.





The following result from \cite{Dinh_Sibony_density} is important for us. We refer the reader to \cite{DinhSibony_pullback} for the discussion on the pull-back of closed positive currents by holomorphic maps.

\begin{lemma}\label{lemma full cohomology minimal dim}
    Suppose that the density h-dimension of $T_1,\ldots,T_m$ is minimal, i.e. it is equal to $n-p_1-\cdots -p_m$. Let $S$ be a density current associated to $T_1,\ldots,T_m$. Then there is a closed positive current $S'$ on $X$ such that $S = \pi_m^*(S')$, and we have
     $$\{S'\} = \{T_1\} \smile  \cdots \smile \{T_m\}.$$
\end{lemma}

This motivates the following definition (see \cite{Viet_Lucas,VietTuanLucas}).

\begin{definition}  \label{def-DSproduct} We say that the \emph{Dinh-Sibony product} $T_1 \curlywedge \cdots \curlywedge T_m$ of $T_1, \ldots, T_m$ is well-defined  if there is only one density current associated to $T_1, \ldots, T_m$ and this current is  the pull-back by $\pi_m$ of a current $S$ on $\Delta_m$. We define $T_1 \curlywedge \cdots \curlywedge T_m := S$. This can be viewed as a current on $X$ as we identify $\Delta_m$ with $X$.
\end{definition}

The notion of Dinh-Sibony products generalizes well-known notions of intersections of currents. We refer to \cite{Dinh_Sibony_density,DNV,Viet_Lucas,VietTuanLucas} for details. We only cite here the following particular case of main results in \cite{VietTuanLucas} (we note that the notion of density currents can be defined for non-compact manifolds as in \cite{VietTuanLucas}). 

\begin{theorem}\label{th-density11agreewithclassical} Let $X$ be a complex manifold. Let $T_1, \ldots, T_m$ be closed positive currents of bi-degree $(1,1)$ on $X$ with locally bounded potentials, and $T$ be a closed positive current of bi-degree $(p,p)$ on $X$. Assume that $m+p\leq n$. Then the Dinh-Sibony product of $T_1, \ldots, T_m,T$ is well-defined and equal to the classical intersection $T_1 \wedge \cdots \wedge T_m \wedge T$. 
\end{theorem}

We will need the following crucial result generalizing the classical comparison of Lelong numbers by Demailly \cite[Corollary 7.9 of Chapter 3]{Demailly_ag}.

\begin{theorem}\label{the-sosanhsoLelongdensity} (\cite[Corollary 1.3]{SuVu-volume-lelong})
For every $x \in X$ and for every density current $S$ associated to $T_1, \ldots, T_{m}$, we have 
        \begin{align*}
            \nu(S, x^{m}) \ge \nu(T_1, x) \cdots \nu(T_m,x),
        \end{align*}
        where $x^{m}=(x, \dots,x) \in \Delta_{m} \subset \mathbb{E}_m$.
\end{theorem}

This result follows from a more general comparison of total density classes established in \cite[Proposition 4.13]{Dinh_Sibony_density}. Recall that the Lelong numbers are preserved under submersions (\cite[Proposition 2.3]{Meo-auto-inter}). Thus, Theorem \ref{the-sosanhsoLelongdensity} implies the following estimate. 

\begin{corollary}\label{cor-sosanhsoLelongdensity} (\cite[Corollary 1.3]{SuVu-volume-lelong})
Assume that the density h-dimension of $T_1, \ldots, T_m$ is minimal. Let $S$ be a density current of $T_1,\ldots,T_m$ and $S'$ be a closed positive current on $X$ such that $\pi_m^*(S')=S$. Then for every $x \in X$, we have 
        \begin{align*}
            \nu(S', x) \ge \nu(T_1, x) \cdots \nu(T_m,x).
        \end{align*}
\end{corollary}

Let $\sigma_V : \widehat{X} \to X$ be the blow-up of $X$ along $V$, and denote by $\widehat{V} :=\sigma_V^{-1}(V)$ the exceptional divisor. Since $\sigma_V: \widehat X\setminus \widehat V \to X \setminus V$ is an isomorphism, one can pull-back the current $T$ to $\widehat X\setminus \widehat V$ and extend trivially to $\widehat X$ to get a current on $\widehat X$. We denote that current by $\sigma_V^{\diamond} (T)$ and call it the \textit{strict transform} of $T$ by $\sigma_V$. Using the inclusion map $\imath: \widehat V \to \widehat X$, we identify currents on $\widehat V$ with their direct image in $\widehat X$ under $\imath$. This induces a natural morphism from $H^*(\widehat V,\R)$ to $H^*(\widehat X,\R)$. By \cite[Lemma 3.14]{Dinh_Sibony_density}, there exists a class $e_V(T)$ in $H^{2p-2}(\widehat{V},\R)$ such that the class $\sigma_V^{*} \{T\} - \{\sigma_V^{\diamond}(T)\}$ in $H^{2p}(\widehat{X},\R)$ is equal to the image of $e_V(T)$ in $H^{2p}(\widehat{X},\R)$ under the above morphism. We have the following cohomology description for total tangent currents.

\begin{proposition}\label{lemma siu descriptuon for tangent currents}(\cite[Proposition 4.12]{Dinh_Sibony_density})
    Identify the projection $\sigma_V : \widehat{V} \to V$ with the projective fiber bundle $\pi:\P(E) \to V$ and denote by $-h_{\P(E)}$ the tautological class of $\pi: \P(E) \to V$. Suppose that $T$ has no mass on $V$. Then we have
    \[\kappa^V(T)|_{\mathbb{P}(E)} = e_V(T) \smile h_{\mathbb{P}(E)} + \pi^*(\{T\}|_V).\]
    In particular, when $V = \{x\}$, we have $\sigma^* \{T\} - \{\sigma^{\diamond}(T)\} = \nu(T,x) \{[H]\}$ where $H$ is a $(n-p)$-dimensional linear subspace of the exceptional divisor $\sigma^{-1}(x) \simeq \C \P^{n-1}$.
\end{proposition}

\section{Lelong numbers of non-pluripolar products}\label{sec sing positive product}

In this section, we study Lelong numbers of non-pluripolar products of closed positive currents. In particular, we show that one can define the notion of Lelong numbers for non-pluripolar positive products of cohomology classes.

We first recall some properties of relative non-pluripolar products introduced in \break \cite{Viet-generalized-nonpluri}. Let $X$ be a compact K\"ahler manifold of dimension $n$. Let $m,p \in \N$ with $m+p\leq n$. Let $T_{1}, \dots , T_{m}$ be closed positive $(1,1)$-currents and let $T$ be a closed positive $(p,p)$-current on $X$. Let $\theta_j$ be a closed $(1,1)$-form in the cohomology class of $T_j$ for $1\leq j\leq m$. Then we can write $T_{j} = \ddc u_{j} +\theta_{j}$ for some function $u_{j} \in \PSH(X,\theta_{j})$. For $k>0$, we set
\[R_{k} := \mathbf{1}_{\bigcap_{j=1}^{m}\{u_{j}>-k\}}\wedge_{j=1}^{m} (\ddc (\max\{u_{j},-k\})+\theta_{j}) \wedge T,\]
which is a positive $(m+p,m+p)$-current on $X$ by \cite[Lemma 3.2]{Viet-generalized-nonpluri}. By \cite[Theorem 3.7]{Viet-generalized-nonpluri}, the masses of $R_{k}$ are uniformly bounded in $k$, and $\{R_{k}\}_{k}$ converge to a closed  positive $(m+p,m+p)$-current as $k \rightarrow \infty$. We denote the limit by 
\[\langle T_{1} \wedge \dotsi \wedge T_{m} \ \dot{\wedge} \ T \rangle,\] and call it \textit{the relative non-pluripolar product} of $T_1,\ldots,T_m$ with respect to $T$. When $T$ is the current of integration along $X$, this notion recovers the usual non-pluripolar product introduced in \cite{BT_fine_87,GZ-weighted,BEGZ}.

Let $R,R'$ be closed positive $(1,1)$-currents in the same cohomology class. Then we can write $R= \ddc u + \theta$ and $R' = \ddc u' + \theta'$, where $\theta,\theta'$ are closed $(1,1)$-forms in the cohomology class of $R,R'$ and $u \in \PSH(X,\theta) , u' \in \PSH(X,\theta')$. We say that $R'$ is less singular than $R$ if $u \leq u' +O(1)$. The relative non-pluripolar product has the following crucial monotonicity property (see \cite[Theorem 4.4]{Viet-generalized-nonpluri}).

\begin{proposition}
\label{lemma compare cohom relative non pluripolar}
    Suppose that for $1\leq j\leq m$, the currents $T_j,T_j'$ are in the same cohomology class and $T_j'$ is less singular than $T_j$. Then we have
    \[\{\langle T_1 \wedge \cdots \wedge T_m \ \dot{\wedge} \ T \rangle\} \leq \{\langle T_1' \wedge \cdots \wedge T_m' \ \dot{\wedge} \ T   \rangle\}.\]
\end{proposition}

We refer the reader to \cite{BEGZ,Lu-Darvas-DiNezza-mono,WittNystrom-mono} for previous works on the monotonicity property.

Let $m,p \in \N$ with $m+p \leq n$. Let $\alpha_{1}, \dots, \alpha_{m}$ be big cohomology classes, and let $T$ be a closed positive $(p,p)$-current. The monotonicity property above allows us to define the non-pluripolar products to $T$ of classes $\alpha_{1} , \dots , \alpha_{m}$, which is denoted and defined as 
\[
    \langle \alpha_{1} \wedge \dotsi \wedge \alpha_{m}\ \dot{\wedge}\ T \rangle := \{\langle T_{\min,\alpha_1} \wedge \dotsi \wedge T_{\min,\alpha_m}\ \dot{\wedge}\ T \rangle\},
\]
where $T_{\min,\alpha_j}$ is a closed positive $(1,1)$ current in $\alpha_{j}$ with minimal singularities. In particular, one gets
$$\langle \alpha_1 \wedge \cdots \wedge \alpha_m\rangle = \{\langle T_{\min,\alpha_1} \wedge \cdots \wedge T_{\min,\alpha_m}\rangle\},$$
see also \cite{BEGZ}.

\begin{theorem}\label{theorem Lelong of relative non-pluripolar product}
    Let $X$ be a compact K\"ahler manifold of dimension $n$. Let $m,p \in \N$ with $m+p\leq n$. For $1\leq j \leq m$, let $T_j,T_j'$ be closed positive $(1,1)$-currents on $X$ such that $T_j,T_j'$ are in the same cohomology class and of the same singularity type. Let $T$ be a closed positive $(p,p)$-current. Then we have  
  $$\nu\big(\langle T_1\wedge \cdots \wedge T_m \ \dot{\wedge} \ T \rangle,x\big) = \nu\big(\langle T_1' \wedge \cdots \wedge T_m' \ \dot{\wedge}  \ T\rangle,x\big)$$
    for every $x \in X$.
\end{theorem}

\begin{proof}

We recall the definition of pull-back of closed positive $(1,1)$-currents by a holomorphic map. Let $T$ be a closed positive $(1,1)$-current on a complex manifold $Y$. Let $\pi:X\to Y$ be a holomorphic map. Suppose that $\pi(X)$ is not contained in the singular set of $T$. Let $\varphi$ be a local potential of $T$. Then we can define the pull-back $\pi^* T$ locally by $dd^c ( \varphi\circ \pi)$. This defines a closed positive $(1,1)$-current on $X$ and we have $\{\pi^* T\} = \pi^* \{T\}$.

Let $\pi: \widehat{X} \to X$ be the blow-up at point $x$. Let $S:= \langle T_1 \wedge \cdots \wedge T_m \ \dot{\wedge} \ T\rangle$ and $S' := \langle T_1' \wedge \cdots \wedge T'_m 
 \ \dot{\wedge} \ T\rangle$. By Proposition~\ref{lemma siu descriptuon for tangent currents}, it suffices to show that $\{S\} = \{S'\}$ and $\{\pi^\diamond(S) \}= \{\pi^\diamond (S')\}$. By Proposition~\ref{lemma compare cohom relative non pluripolar}, it is clear that $\{S\} = \{S'\}$. We consider the relative non-pluripolar product $\langle \pi^* T_1 \wedge \cdots \wedge \pi^* T_m \ \dot{\wedge} \ \pi^\diamond(T)\rangle $. Since $T$ puts no mass on $\{x\}$, $\pi^\diamond(T)$ puts no mass on $\pi^{-1}(x)$. Moreover, we have $\pi$ is a biholomorphic map outside $\pi^{-1}(x)$. This implies that $\langle \pi^* T_1 \wedge \cdots \wedge \pi^* T_m \ \dot{\wedge} \ \pi^\diamond(T)\rangle  = \pi^\diamond (\mathbf{1}_{X\setminus\{x\}} S) = \pi^\diamond (S)$ as $T$ puts no mass on $\{x\}$. A similar statement holds for $S'$. Moreover, we have $\pi^* T_j$ and $\pi^* T_j'$ are in the same cohomology class and of the same singularity type. Then, by Proposition~\ref{lemma compare cohom relative non pluripolar}, $\{\pi^\diamond(S) \} = \{\pi^\diamond (S')\}$. The proof is complete.
\end{proof}

Let $\alpha_1,\ldots, \alpha_m$ be big cohomology $(1,1)$-classes. 
We define the Lelong number of $\langle \alpha_1 \wedge \cdots \wedge \alpha_m \rangle$ at $x\in X$ as
$$\nu\big(\langle \alpha_1 \wedge \cdots \wedge \alpha_m \rangle, x\big):= \nu\big(\langle T_{\min,\alpha_1} \wedge \cdots \wedge T_{\min,\alpha_m} \rangle, x \big).$$
By Theorem \ref{theorem Lelong of relative non-pluripolar product}, this definition  is independent of the choice of currents with minimal singularities. 
We also have the following theorem concerning total density classes.

\begin{theorem} \label{th-densityclassbangnhau}
For $1\leq j\leq m$, let $T_j,T_j'$ be closed positive $(1,1)$-currents on $X$ such that $T_j,T_j'$ are in the same cohomology class and of the same singularity type. Put 
\[S:= \langle T_1 \wedge \cdots \wedge T_m \rangle \text{ and }S':= \langle T_1' \wedge \cdots \wedge T_m' \rangle.\] Let $T$ be a closed positive $(p,p)$-current on $X$. Then the total density classes $\kappa(S,T)$ and $\kappa(S',T)$ are equal. 
\end{theorem}

By taking $T$ to be the Dirac mass at a point $x\in X$ in Theorem \ref{th-densityclassbangnhau}, using Proposition \ref{lemma siu descriptuon for tangent currents},  we recover Theorem \ref{theorem Lelong of relative non-pluripolar product} in the special case where $T=[X]$, which is enough to define Lelong numbers of positive products. We chose to present Theorem \ref{theorem Lelong of relative non-pluripolar product} first because it does not involve density currents, thus it is easier to explain the idea of the proof.

\begin{proof} We need to show that $\kappa^{\Delta_2}(S \otimes T) = \kappa^{\Delta_2}(S'\otimes T)$. Let $\pi_1 :X\times X \to X$ and $\pi_2 : X\times X \to X$ be the projection to the first and second variables. Observe that if  $T_j$ is of bounded potential for $j=1,\ldots,m$, then we have
$$S\otimes T = \pi_1^*(T_1) \wedge \cdots \wedge \pi_1^*(T_m) \wedge \pi_2^*(T)$$
because of the continuity of Monge-Amp\`ere operators and regularization. In general, we write locally $T_j= \ddc u_j$ for $1 \le j \le m$ on a local chart $U$ in $X$ and put 
$$u_{j,k}:= \max\{u_j, -k\}, \quad S_{(k)}:= \ddc u_{1,k} \wedge \cdots \wedge  \ddc u_{m,k}.$$
As above, we have 
$$S_{(k)} \otimes T = \pi_1^* (\ddc u_{1,k}) \wedge \cdots \wedge \pi_1^* (\ddc u_{m,k}) \wedge \pi_2^* (T).$$
It follows that 
\begin{align*}
(\mathbf{1}_{\bigcap_j \{u_j>-k\}} S_{(k)}) \otimes T &= \mathbf{1}_{\bigcap_j \pi_1^{-1}\{u_j>-k\}} (S_{(k)} \otimes T)\\
&=  \mathbf{1}_{\bigcap_j \pi_1^{-1}\{u_j>-k\}}\pi_1^*( \ddc u_{1,k}) \wedge \cdots \wedge \pi_1^* (\ddc u_{m,k} )\wedge \pi_2^* (T).
\end{align*}
Letting $k\to \infty$ gives
$$S\otimes T = \langle \pi_1^*(T_1) \wedge \cdots \wedge \pi_1^*(T_m) \ \dot{\wedge} \ \pi_2^*(T) \rangle.$$
Analogously, we get
$$ S'\otimes T = \langle \pi_1^*(T_1') \wedge \cdots \wedge \pi_1^*(T_m') \ \dot{\wedge} \ \pi_2^*(T) \rangle.$$

Let $\sigma_{\Delta_2} :\widehat{X\times X} \to X\times X$ be the blow-up along the diagonal $\Delta_2$. Since $\pi_2^*(T)$ puts no mass on $\Delta_2$, argue similarly as in the proof of Theorem~\ref{theorem Lelong of relative non-pluripolar product}, we have
\[\langle \sigma_{\Delta_2}^*(\pi_1^*(T_1)) \wedge \cdots \wedge \sigma_{\Delta_2}^*(\pi_1^*(T_m) ) \ \dot{\wedge} \ \sigma_{\Delta_2} ^{\diamond}(\pi_2^*(T)) \rangle = \sigma_{\Delta_2} ^{\diamond} (S\otimes T). \]
Similarly, we also get
\[\langle \sigma_{\Delta_2}^*(\pi_1^*(T_1')) \wedge \cdots \wedge \sigma_{\Delta_2}^*(\pi_1^*(T_m') ) \ \dot{\wedge} \ \sigma_{\Delta_2} ^{\diamond}(\pi_2^*(T)) \rangle = \sigma_{\Delta_2} ^{\diamond} (S'\otimes T). \]
By Proposition~\ref{lemma siu descriptuon for tangent currents} and Proposition~\ref{lemma compare cohom relative non pluripolar}, we obtain the desired assertion.
    \end{proof}

\begin{remark} We can easily construct a closed positive $(1,1)$-current $T$ in a K\"ahler class $\alpha$ such that $T$ is smooth outside a given point $x_0$ and $\nu(T,x_0)>0$. It follows that $\langle T^k \rangle =T^k$ (both sides are equal outside $x_0$ and have no mass on $x_0$) for $1 \le k \le n-1$. Consequently, we have $\nu(\langle T^k\rangle, x_0)= \nu(T^k,x_0)>0$. From this we see that the conclusion of Theorem \ref{cor-divisorpartialderivitve} concerning the Lelong number of $\langle T_{\min,\alpha}^k \rangle$ does not remain true if one replaces $T_{\min,\alpha}$ by an arbitrary current in $\alpha$.   
\end{remark}

We now discuss an extension of Conjecture \ref{conj-lelong}. We recall the notion of nefness in higher codimension introduced in \cite{XiaojunWu-nefnessHigherDim}. Let $\alpha$ be a big class in $H^{1,1}(X,\R)$. We say that $\alpha$ is \textit{nef in codimension $k$} if for every irreducible analytic subset $Z \subset X$ of codimension at most $k$, we have $\nu(\alpha,Z) = 0$. It is known that the notion of modified-nef introduced in \cite{Boucksom_anal-ENS} is equivalent to nef in codimension $1$. Moreover, $\alpha$ is nef in codimension $n-1$ if and only if $\alpha$ is nef. We propose the following generalized conjecture of Conjecture~\ref{conj-lelong}.

 \begin{conjecture}\label{conj nef codim k}
Let $X$ be a compact K\"ahler manifold of dimension $n$ and $k\in \N$ such that $1\le k \le n-1$.   Let $\alpha \in H^{1,1}(X,\R)$ be a big class such that $\alpha$ is nef in codimension $k$. Then $\nu( \langle \alpha^{n-k}  \rangle,x) = 0$ for every $x\in X$.
    \end{conjecture}

\section{Comparison of Lelong numbers}\label{sec proof of main results}

In this section, we compare various notions of intersection of currents that will be used in the proof of the main theorems. Let $X$ be a compact K\"ahler manifold of dimension $n$ and $\omega$ be a K\"ahler form on $X$. We first prove the following theorem that compares the Dinh-Sibony product and the classical product, this can be viewed as a different version of Theorem~\ref{th-density11agreewithclassical}.

\begin{theorem}\label{theorem density for bounded potential mix version} Let $m,p \in \N$ with $m+p \le n$. Let $T_1,\ldots,T_m$ be closed positive $(1,1)$-currents with local bounded potentials on $X$ for $1\leq j\leq m$.   Let $T$ be a closed positive $(p,p)$-current. Then, the Dinh-Sibony product $(T_1 \wedge \cdots \wedge T_m) \curlywedge T$ is well-defined and equal to the classical product $T_1\wedge \cdots \wedge T_m \wedge T$.
\end{theorem}

We emphasize a subtle difference with Theorem \ref{th-density11agreewithclassical}: the present statement concerns the wedge product of the $(m,m)$-current $(T_1\wedge \cdots\wedge T_m)$ with $T$ rather than the simultaneous product of the $(1,1)$-currents $T_1,\ldots,T_m$ with $T$. The compactness of $X$ is indeed not necessary. For simplicity, we state the result for compact manifolds.


We will use ideas from \cite{Viet-density-nonpluripolar} to prove Theorem \ref{theorem density for bounded potential mix version}. We will need  the following auxiliary lemmas.

Let $Y$ be a compact complex manifold of dimension $N$ and $V$ be a smooth hypersurface on $Y$. Let $\pi : E\to V$ be the normal bundle of $V$ in $Y$. Let $\omega$ be a Hermitian form on $Y$. Let $\varphi$ be a potential of $[V]$, then $\ddc \varphi -[V]$ is a smooth form.   Observe that there exists a constant $c_0>0$ big enough such that $\varphi$ is  $c_0 \omega$-psh on $Y$.

Consider a local chart $U$ of $Y$ which is identified with the polydisc $\D^N$ such that $V\cap U$ is equal to $\D^{N-1} \times \{0\}$. We will use in this polydisc the standard coordinates $z= (z_1,\ldots,z_N)$. Denote by $A_\lambda$ the map $(z_1,\ldots,z_N) \mapsto (z_1,\ldots,z_{N-1},\lambda z_N)$ where $\lambda \in \C^*$. Thus for a form $\Phi$ with support on $\D^N$, we see that  $A_\lambda^*(\Phi)$ has support on $\D^{N-1} \times |\lambda|^{-1} \D$.

We have the following lemma, which is a direct consequence of \cite[Lemma 2.11]{Dinh_Sibony_density} (see also \cite[Lemma 2.6]{Viet-density-nonpluripolar}).

\begin{lemma}\label{lemma bound *norm}
    Let $K$ be a compact subset of $\D^N$ and $p \in \Z_+$. Let $\Phi$ be a positive smooth $(p,p)$-form on $\D^N$. Then there exists a sequence of quasi-psh functions $\varphi_\lambda$ such that $\supp(\varphi_\lambda) \Subset \supp (\varphi)$, $\varphi_\lambda$ decreases to $\varphi$, and $A_\lambda^*(\Phi) \le C (\omega + dd^c (\varphi_\lambda))\wedge \omega^{p-1}$  on $ K \cap (\D^{N-1} \times |\lambda|^{-1} \D)$ where $C$ is a constant that does not depend on $\lambda$.
\end{lemma}

Let $U$ be an open subset in $\C^N$, and $R$ be a closed positive current of bi-dimension $(m,m)$ on $U$. Given a Borel subset $K$ of $U$. We define the relative capacity of $K$ with respect to the current $R$ in $U$ by
\[\capa_R(K,U):= \sup \Big\{ \int_K (dd^c u)^m \wedge R: u \text{ is psh on }U \text{ and }0\leq u\leq 1\Big\}.\]

Recall that a set $A$ is called \textit{locally complete pluripolar} if locally $A = \{u = -\infty\}$ for some psh function $u$. We have the following lemma (see \cite[Lemma 2.1]{Viet-generalized-nonpluri}).

\begin{lemma} \label{lemma capT = 0 on pluripolar set}
 Suppose that $R$ puts no mass on a locally complete pluripolar set $A$. Then we have $\capa_R(A,U) = 0$.
\end{lemma}

We also need the following quasi-continuity type lemma, which allows us to deal with a family of currents (see \cite[Theorem 2.4]{Viet-generalized-nonpluri}).

\begin{lemma}\label{lemma strong quasi-continuity}
Let $S$ be a closed positive current and $u$ be a bounded psh function on $U$. Let $v_k,v$ be psh functions such that $v_k \geq v$ for every $k$ and $v_k $ converges to $v$ in $L^1_{\mathrm{loc}}$ as $k \to \infty$. Suppose that $v_k,v$ are locally integrable with respect to current $S$ and consider the family of currents $R_k:= dd^c v_k \wedge S$. Then, for every $\varepsilon > 0$, there exists an open subset $U_1 $ of $U$ such that $u$ is continuous on $U \setminus U_1$ and $\capa_{R_k} (U_1,U) < \varepsilon$ for every $k$.
\end{lemma}

\begin{proof}[Proof of Theorem \ref{theorem density for bounded potential mix version}]

We prove by induction on $m$. The case $m=1$ follows from \break Theorem \ref{th-density11agreewithclassical}. Suppose that the theorem is true for $m-1$ with $m>1$. Since this is a local problem, we will work on a local chart $U$. We assume that $T_j = dd^c u_j$ on $U$ where $u_j$ is a non-negative psh function. Consider the positive current $u_1 dd^c u_2 \wedge \cdots \wedge dd^c u_m \otimes T$ on $X^2$. Let $\Delta$ be the diagonal of $X^2$ and $\pi: E \to \Delta$ be the normal bundle of $\Delta$ in $X^2$.
   \vspace{0.5cm}

\noindent        
\textbf{Claim.} The tangent current of $u_1 dd^c u_2 \wedge \cdots \wedge dd^c u_m \otimes T$ along the diagonal $\Delta$ is 
\[\pi^* (u_1 dd^c u_2 \wedge \cdots \wedge dd^c u_m \wedge T).\] 

Indeed, we have to show that
    \begin{equation}\label{eq density uddc u}\lim_{\lambda \to \infty} (A_\lambda)_* \tau_* (u_1 dd^c u_2 \wedge \cdots \wedge dd^c u_m \otimes T) = \pi^* (u_1 dd^c u_2 \wedge \cdots \wedge dd^c u_m \wedge T )\end{equation}
    for every holomorphic admissible map $\tau$. We note that by induction assumption, we have
    \begin{equation}\label{eq density ddc u wedge}\lim_{\lambda \to \infty} (A_\lambda)_* \tau_* (dd^c u_2 \wedge \cdots \wedge dd^c u_m \otimes T) = \pi^* (dd^c u_2 \wedge \cdots \wedge dd^c u_m \wedge T ).\end{equation}
    It is clear that when $u_1$ is continuous, \eqref{eq density uddc u} follows from \eqref{eq density ddc u wedge}.

Let $Y:= X^2$ and $\sigma : \widehat{Y} \to Y$ be the blow-up of $Y$ along $V:= \Delta$. This step allows us to consider the hypersurface situation as in Lemma~\ref{lemma bound *norm}.
    Let $\widehat{V}$ be the exceptional divisor. Let $\widehat{E}$ be the blow-up of $E$ along $V$. We can identify $\widehat{E}$ with the normal bundle of $\widehat{V}$ in $\widehat{Y}$. Let $\pi_{\widehat{E}} : \widehat{E} \to \widehat{V}$ be the natural projection. We pick a K\"ahler form $\widehat \omega$ on $\widehat{Y}$.

Let $p_1,p_2$ be the projections of $X^2$ to the first and second component respectively. Since $p_1 \circ \sigma,p_2 \circ \sigma$ are submersions, we can define
    \[\widehat{u}_j:= (p_1 \circ \sigma)^* u_j, \quad \widehat{T}:= (p_2 \circ \sigma)^* T.\]
    It is clear that $\widehat{u}_j$ is bounded for $1\leq j\leq m$.    Put 
    $$R:= (T_2 \wedge \cdots \wedge T_m) \otimes T,\quad Q:= T_2 \wedge \cdots \wedge T_m \wedge T.$$
     By induction hypothesis, $\pi^* Q$ is the unique tangent current of $R$ along $V$. Let $\widehat{R}$ be the strict transform of $R$ by $\sigma$. By \cite[Lemma 4.7]{Dinh_Sibony_density}, the unique tangent current of $\widehat{R}$ along $\widehat{V}$ is $\pi^*_{\widehat{E}} (\sigma|_{\widehat{V}})^* Q$ (we note here that \cite[Lemma 4.7]{Dinh_Sibony_density} was stated for closed positive currents but its proof works actually for  positive currents).  Thus, for every admissible map $\widehat{\tau}$ from an open neighborhood of $\widehat{V}$ in $\widehat{Y}$ to $\widehat{E}$ and for every continuous function $f$ on $\widehat{Y}$, we have
    \begin{equation}\label{eq lim a lamda f R}
    \lim_{\lambda \to \infty} (A_\lambda)_* (\widehat{\tau})_* (f\widehat{R}) = \pi^*_{\widehat{E}} ((f|_{\widehat{V}})(\sigma|_{\widehat{V}})^* Q).
    \end{equation}
If $u_1$ is continuous, then the desired assertion follows immediately from (\ref{eq lim a lamda f R}) by applying it to $f= \widehat {u}_1$. The main difficulty, as in the proof of the main result in \cite{Viet-density-nonpluripolar}, is to reduce the general situation to the case where $u_1$ is continuous.

 Let $\varphi$ and $\varphi_{\lambda}$ be quasi-psh functions associated with $\widehat{V}$ as in Lemma~\ref{lemma bound *norm}. Since the fibers of $p_2\circ \sigma$ intersect $\widehat{V}$ properly, we have $\varphi$ is locally integrable with respect to $\widehat{T}$. This also implies that $\widehat{T}$ puts no mass on $\widehat{V}$.

    By Lemma~\ref{lemma capT = 0 on pluripolar set}, we have $dd^c \widehat{u}_2 \wedge \cdots \wedge dd^c \widehat{u}_m \wedge \widehat{T}$ puts no mass on $\widehat{V}$. Since $\widehat{R}$ also puts no mass on $\widehat{V}$ and $\sigma$ is an isomorphism on $\widehat{Y}\setminus \widehat{V}$, we have \begin{equation}\label{eq ddc u wedge T = R blow up}
        dd^c \widehat{u}_2 \wedge \cdots \wedge dd^c \widehat{u}_m \wedge \widehat{T} = \widehat{R}.
    \end{equation}
    Since $\varphi$ is locally integrable with respect to $\widehat{R}$, \eqref{eq ddc u wedge T = R blow up} implies that $\varphi_\lambda$ is also locally integrable with respect to $\widehat{R}$. Therefore, we can apply Lemma~\ref{lemma strong quasi-continuity} for the family of currents $(dd^c \varphi_{\lambda} + \widehat{\omega}) \wedge \widehat{R}$.

    Consider a local chart $\widehat{U}$ of $\widehat{Y}$ small enough such that $\widehat{E}$ is trivializable on $\widehat{V} \cap \widehat{U}$. Then we can consider the identity map as the local holomorphic admissible map. Let $\Phi$ be a smooth positive form of bi-degree $(2n-m-p+1,2n-m+1)$. By Lemma~\ref{lemma bound *norm}, we can bound
    \begin{equation}\label{eq bound A^* lambda Phi}
        A_\lambda^* \Phi \lesssim(dd^c \varphi_{\lambda} + \widehat{\omega}) \wedge \widehat{\omega}^{2n-m-p}.
    \end{equation}
    Using \eqref{eq bound A^* lambda Phi} and applying Lemma~\ref{lemma strong quasi-continuity} for $\widehat{u_1}$, we deduce that, for any $\varepsilon > 0$, there exists an open subset $\widehat{U}_1$ of $\widehat{U}$ such that
    \begin{equation}\label{eq strong quasi-continuity}\capa_{A_\lambda^* \Phi \wedge \widehat{R}} (\widehat{U}_1,\widehat{U}) < \varepsilon
    \end{equation}
    for every $\lambda$ and $\widehat{u}_1$ is continuous on $\widehat{U}\setminus \widehat{U}_1$. Moreover, we can pick an open subset $W$ of $X \simeq V$ such that 
    \begin{equation}\label{eq quasi-continuity for T}
        \capa_T(W) < \varepsilon
    \end{equation}
    and $u_1$ is continuous on $X \setminus W$. Put $B:= (\widehat{U} \setminus \widehat{U}_1) \cup (p_1 \circ \sigma)^{-1}(X\setminus W)$. Then $\widehat{u}_1$ is continuous on $B$.

    Pick a continuous function $\widehat{u}'_1$ such that $\|\widehat{u}'_1\|_{L^\infty} \leq \|u_1\|_{L^\infty}$ and $\widehat{u}'_1 = \widehat{u_1}$ on
    $B \cap \widehat{U}$. By \eqref{eq strong quasi-continuity}, we have
    \begin{equation}\label{eq bound mass}\|(A_\lambda)_* ( \widehat{u_1} \widehat{R}) - (A_\lambda)_*(\widehat{u}'_1\widehat{R})\| \lesssim \varepsilon,\end{equation}
    uniformly in $\lambda$, where the mass is taken on a compact subset of $\widehat{U}$. Moreover, by \eqref{eq quasi-continuity for T}, we have
    \begin{equation}\label{eq bound mass 2}
        \|\pi_{\widehat{E}}^* ((\widehat{u}'_1|_{\widehat{V}} -\widehat{u}_1|_{\widehat{V}})(\sigma|_{\widehat{V}})^* Q )\| \lesssim \varepsilon,
    \end{equation}
   where the mass is taken on a compact subset of $\widehat{U}$. Combining \eqref{eq lim a lamda f R}, \eqref{eq bound mass}, and \eqref{eq bound mass 2}, we infer the claim.

    Since our problem is local, we can assume that $\tau$ is identity map in \eqref{eq density uddc u}. Taking $dd^c$ in the first variable of both sides in \eqref{eq density uddc u}, the proof is complete.
\end{proof}

The next theorem compares relative non-pluripolar product and density currents.

\begin{theorem}\label{theorem main compare Lelongphay}
Let $T_1,\ldots,T_{m}$ be closed positive $(1,1)$-currents on $X$ for $1\leq m\leq n-1$. Let $T$ be a closed positive $(1,1)$-current on $X$ and $R:=\langle T_1\wedge \cdots \wedge T_m \rangle$. Let $S$ be a density current of $T$ and $R$. Then $S$ is of minimal h-dimension, and we have 
$$\langle \wedge_{j=1}^{m} T_j \ \dot{\wedge} \ T\rangle  \le S',$$
where $S'$ is a closed positive current on $X$ such that $\pi_2^*(S') =S$.
\end{theorem}

\begin{proof}
Let the notations be as in the proof of Theorem \ref{theorem density for bounded potential mix version}.
Since this is a local problem, we will work locally on a chart $U$. On this chart, we write $T_{j}$ as $\ddc u_{j}$ where $u_j$ is a negative psh function on $U$. For $k \in \N$, we set 
    \[
        u_{j,k} :=  \max\{u_{j},-k\} \mbox{ and } T_{j,k} := \ddc u_{j,k}.
    \]
    Since $T_{j,k}$ is of bounded potential, the classical product $\wedge_{j=1}^{m} T_{j,k}$ is well-defined. We set 
    \[
        P_{k} := \wedge_{j=1}^{m} T_{jk} \otimes T \mbox{ and } P := \langle \wedge_{j=1}^{m} T_{j} \rangle \otimes T.
    \]
We also define
    \[
        \psi :=\sum_{j=1}^{m} u_{j} \circ p_1 \mbox{ and }\psi_{k}  = k^{-1} \max\{\psi, -k\}. 
    \]
   It is clear that $\psi_{k}+1=0$ on $\bigcup_{j=1}^{m} p_{1}^{-1} \{u_{j} \leq -k\}$. This implies
    \begin{align}
                   -\psi_{k}P  &= -(\psi_{k}+1) P + P \label{equ1} \\
                   &= -(\psi_{k}+1) \mathbf{1}_{\bigcap_{j=1}^{m} \{u_{j}>-k\} } \langle \wedge_{j=1}^{m} T_{j} \rangle \otimes T + P \notag \\
                   &= -(\psi_{k}+1)P_{k} + P. \notag
    \end{align}

     Let $(\lambda_l)$ be a defining sequence for the density current $S$. Let $P_{k,\infty}$, $P_{\infty}$ be tangent currents of $-\psi_{k}P$, $P$ along $\Delta \subset X^{2}$ (defined using the sequence $(\lambda_l)$) respectively. We have $P_\infty = S$. By the Claim in the proof of Theorem \ref{theorem density for bounded potential mix version}, we get that the tangent current of $(\psi_{k}+1)P_{k}$ equals $\pi^{*} ((\rho_{k}+1) \wedge_{j=1}^{m}T_{jk} \wedge T)$, where  $\rho_{k}$ is the restriction of $\psi_{k}$ on $\Delta$. Now, by taking the tangent current of both sides of \eqref{equ1}, we get 
    \begin{equation}
    \label{equ2}
        P_{k,\infty} = P_{\infty} - \pi^{*} ((\rho_{k}+1) \wedge_{j=1}^{m}T_{jk} \wedge T).
    \end{equation}
    Since $\rho_{k}+1=0$ on $\bigcup_{j=1}^{m} \{u_{j} \leq -k\}$, we obtain
    \[
        (\rho_{k}+1) \wedge_{j=1}^{m}T_{jk} \wedge T = (\rho_{k}+1) \langle \wedge_{j=1}^{m} T_{j} \ \dot{\wedge} \ T \rangle,
    \]
    which converges to $\langle \wedge_{j=1}^{m} T_{j} \ \dot{\wedge}\ T \rangle$ as $k \rightarrow \infty$. This, combined with \eqref{equ2}, implies
    \begin{equation}\label{eq compare relative with density}
        \pi^{*}(\langle \wedge_{j=1}^{m} T_{j} \ \dot{\wedge} \ T \rangle) \leq   P_{\infty}.
    \end{equation}

     Since $\langle \wedge_{j=1}^m T_j\rangle$ puts no mass on the pluripolar set $\{x:\nu(T,x)>0\}$, the h-dimension between $\langle \wedge_{j=1}^m T_j\rangle$ and $T$ is strictly smaller than $n-m$ by \cite[Proposition 5.7]{Dinh_Sibony_density}. Since the h-dimension is at least $n-m-1$ (by a bi-degree reason and the fact that $T$ is of bi-degree $(1,1)$), it must be equal to $n-m-1$. Therefore, there exists a closed positive current $P_\infty'$ on $X$ such that $P_\infty = \pi^* (P_\infty')$. This combines with \eqref{eq compare relative with density} to imply that $\langle \wedge_{j=1}^{m} T_{j} \ \dot{\wedge} \ T \rangle \leq P_\infty'$ as desired.
\end{proof}

\section{Proof of main results}

In this last section, we prove our main theorems.

\begin{proof}[Proof of Theorem \ref{theorem main compare Lelong}]
Let the notations be as in  Theorem \ref{theorem main compare Lelongphay}. The first part of Theorem \ref{theorem main compare Lelong} follows directly from 
Theorem \ref{theorem main compare Lelongphay}.

We now prove the second part. Since $T$ has no mass on point sets, by Lemma  \ref{lemma capT = 0 on pluripolar set}, the measure $\langle \wedge_{j=1}^{n-1} T_j\ \dot{\wedge}\ T \rangle$ has no mass on point-sets on $X$.   Corollary \ref{cor-sosanhsoLelongdensity} implies that
$$S' \ge  \sum_{x \in X} \nu(R,x) \nu(T,x) \delta_x,$$
which, combined with Theorem \ref{theorem main compare Lelongphay}, gives
$$S' \ge  \langle \wedge_{j=1}^{n-1}T_j \ \dot{\wedge} \ T \rangle +  \sum_{x \in X} \nu(R,x) \nu(T,x) \delta_x.$$
Therefore, if the mass of $S'$ and $\langle \wedge_{j=1}^{n-1} T_j \ \dot{\wedge}\ T \rangle$ are equal, one must have $\nu(R,x)\cdot  \nu(T,x)=0$ for every $x\in X$. The proof is complete.
\end{proof}

\begin{corollary} \label{cor-divisorpartialderivitve2} Let $X$ be a compact K\"ahler manifold of dimension $n$. Let $\alpha$ be a big cohomology class on $X$ and $D$ be a real effective divisor on $X$. Assume that 
\begin{align}\label{ine-totaldensityDTmin2}
 \langle \alpha^{n-1}\rangle|_{X|D}= \langle \alpha^{n-1}\rangle \smile \{D\}.
\end{align}
Then we have $\nu(\langle \alpha^{n-1}\rangle, x)=0$ for every $x\in \supp (D)$. In particular, if $X$ is projective, then we have 
$$\nu(\langle \alpha^{n-1}\rangle, x)=0$$
 for every $x\in X$.  
\end{corollary}

We note that we do not know if the vanishing of Lelong numbers of $\langle \alpha^{n-1}\rangle$ can imply the equality~\eqref{ine-totaldensityDTmin2}.

\begin{proof}
We recall that  by \cite[Lemma 4.5 and Proposition 4.6]{Vu_derivative}, we have
 $$\langle \alpha^{n-1}\rangle|_{X|D}=\langle T_{\min,\alpha}^{n-1} \ \dot{\wedge} \ [D] \rangle.$$
The desired first assertion now follows directly from Theorem \ref{theorem main compare Lelong} and the fact that $\nu([D],x)>0$ for every $x\in D$. 

Consider now the case where $X$ is projective. Let $x \in X$, and $D$ be a smooth hyperplane section passing through $x$.  By \cite{WittNystrom-duality}, we have (\ref{eq-BDPP}) for $\gamma= \{D\}$. This combined with \cite{WittNystrom-deform,Vu_derivative} gives
$$\langle \alpha^{n-1}\rangle|_{X|D}= \langle \alpha^{n-1} \rangle \smile \{D\}$$
because both are equal to $1/n$ times the derivative of $\volume(\alpha+ t \{D\})$ at $t=0$. The first assertion now implies $\nu(\langle \alpha^{n-1} \rangle, x)=0$. This finishes the proof.
\end{proof}

We can rephrase Corollary \ref{cor-divisorpartialderivitve2} as follows: if the equality \eqref{eq-BDPP} holds for $\gamma=\{D\}$, then $\nu(\langle \alpha^{n-1}\rangle, x)=0$ for every $x\in \supp (D)$. We also have the following variant.

\begin{theorem} \label{the-main-Zariski} Let $X$ be a compact K\"ahler manifold of dimension $n$. Let $\alpha$ be a big cohomology class on $X$. Suppose that the conjectural identity \eqref{eq-BDPP} holds for $X$.  Then the Lelong number of $\langle \alpha^{n-1} \rangle$ is zero everywhere on $X$. 
\end{theorem}

\begin{proof}
Let $x_0\in X$. 
Let $T = \omega + dd^c \varphi$ be a closed positive $(1,1)$-current cohomologous to $\omega$ such that $T$ is smooth outside $x_0$ and $\nu(T,x_0)>0$ (such a current can be constructed locally using the function $\varepsilon \log |x-x_0|$ for some constant $\varepsilon>0$ small enough). The desired assertion will follow from  Theorem~\ref{theorem main compare Lelong} if  we can show that 
\begin{align}\label{eq-cancmTminalhapro}
\{\langle T_{\min,\alpha}^{n-1}\rangle\} \smile \{T\} = \{ \langle T_{\min,\alpha}^{n-1} \ \dot{\wedge} \ T \rangle \}.
\end{align}

 By Demailly's approximation theorem (see \cite{Demailly_regula_11current}) and Hironaka's theorem, there are closed positive $(1,1)$-currents $T_k$ of analytic singularities and smooth modifications $\pi_k:X_k \to X$ such that the sequence of potentials of $T_k$ is decreasing and 
 $$T_k \to T_{\min,\alpha},\quad \pi_k^* (T_k) = \omega_k + [D_k]$$
  for some semi-positive form $\omega_k$ and effective $\mathbb{Q}$-divisor $D_k$ on $X_k$. We also have that $\nu(T_k,x) \to \nu(T_{\min,\alpha},x)$ as $k\to \infty$ for every $x \in X$. It follows that 
\begin{align} \label{conver-TmTk} 
\langle T_k^m \rangle \to \langle T_{\min,\alpha}^m \rangle, \quad \langle T_k^m \ \dot{\wedge}\ T\rangle  \to \langle T_{\min,\alpha}^m \ \dot{\wedge}\ T \rangle
  \end{align}
  weakly as $k \to \infty$ for every $1 \le m \le n-1$ (see \cite[Theorem 4.9]{Viet-generalized-nonpluri}), and $\pi_k(\supp (D_k))$ is contained in the set $\{x\in X:\nu(T_{\min,\alpha},x)>0\}$.

   Since $T$ has no mass on proper analytic subsets in $X$ (in particular $\pi_k(\supp (D_k))$) and $\pi_k$ is biholomorphic outside $\supp (D_k)$, we see that 
$$\int_{X} \langle T_k^{n-1} \ \dot{\wedge} \ T \rangle=\int_{X \setminus \pi_k(\supp (D_k))} \langle T_k^{n-1} \ \dot{\wedge} \ T \rangle= \int_{X_k \setminus  \supp (D_k)} \omega_k^{n-1} \wedge \pi_k^*(T).$$  
It follows that  
$$\int_X \langle T_{\min,\alpha}^{n-1} \ \dot{\wedge} \ T \rangle = \lim_{k\to \infty}\int_{X} \langle T_k^{n-1} \ \dot{\wedge} \ T \rangle =\lim_{k\to \infty} \int_{X_k \setminus \supp (D_k)} \omega_k^{n-1} \wedge \pi_k^*(T).$$

On the other hand, we have
\begin{align*}
\int_X \{\langle T_{\min,\alpha}^{n-1}\rangle\} \smile \{T\} &= \lim_{k\to \infty}\int_X \{\langle T_k^{n-1}\rangle\} \smile \{T\}\\
&= \lim_{k\to \infty}\int_{X_k} \pi_k^*\{ \langle T_k^{n-1}\rangle \}\smile \{\pi^*_k (T)\}\\
&= \lim_{k\to \infty} \int_{X_k} \omega_k^{n-1} \wedge \pi_k^*(T).
\end{align*}  
Hence,  the desired equality (\ref{eq-cancmTminalhapro}) is equivalent to
\begin{align}\label{eq-limksuppD_k}
\lim_{k\to \infty} \int_{\supp (D_k)} \omega_k^{n-1}\wedge  \pi_k^* (T) = 0.
\end{align}

Write $\pi_k^* (T) = R_k+ R_k'$, where $R_k$ is a closed positive $(1,1)$-current supported on $\supp (D_k)$ and the closed positive current $R'_k$ has no mass on $\supp (D_k)$. Since $T$ is smooth outside $x_0$, we see that 
\begin{equation*}
\supp (R_k) \subset \pi_k^{-1}(x_0) \cap \supp (D_k)
\end{equation*}
which is empty if $\nu(T_{\min,\alpha},x_0)=0$. 
   Observe that 
$$R_k = \sum_D \nu(R_k,D)[D] \le \sum_{D} \nu(\pi_k^*(T), D) [D],$$
where the sums run over all irreducible components $D$ in $D_k$.

 We note that 
\begin{align}\label{ine-uocluongEktichphan}
\int_{\supp (D_k)} \omega_k^{n-1}\wedge \pi_k^* (T)&= \int_{\supp (D_k)} \omega_k^{n-1}\wedge R_k\\ 
\nonumber &= \int_{\supp (R_k)} \omega_k^{n-1}\wedge R_k\\ \nonumber
&= \int_{\pi_k^{-1}(x_0)\cap \supp (D_k)} \omega_k^{n-1}\wedge R_k 
\end{align} 
which is zero if $\nu(T_{\min,\alpha},x_0)=0$ by the above arguments.

Consider now the case where $\nu(T_{\min,\alpha},x_0)>0$ and $k$ big enough such that 
$$\nu(T_k,x_0) \ge \delta:=\frac{1}{2}\nu(T_{\min,\alpha},x_0)>0.$$
  By the construction of $T$,  there is a constant $\epsilon>0$ such that $(\epsilon \delta) T$ is less singular than $T_k$. Thus, the current $(\epsilon \delta) \pi_k^*(T)$ is less singular than $\pi_k^* (T_k)$. Consequently, one gets 
\begin{align} \label{ine-RkchanDk}
R_k= \mathbf{1}_{\supp (D_k)} \pi_k^*(T) \le  (\epsilon\delta)^{-1}\mathbf{1}_{\supp (D_k)} \pi_k^*(T_k)= (\epsilon \delta)^{-1}[D_k].
\end{align}
By \eqref{eq-BDPP} and \cite[Appendix A]{WittNystrom-duality}, we know that $\volume(\alpha)= \langle \alpha^{n-1} \rangle \smile \alpha$. Pulling back both sides by $\pi_k$ and using (\ref{conver-TmTk}) give 
\begin{align*}\lim_{k\to \infty} \int_{X_k}\omega_k^{n-1} \wedge [D_k]&= \lim_{k\to \infty} \int_{X_k} \omega_k^{n-1} \wedge \pi_k^* T_k - \int_{X_k} \omega_k^{n}\\&=\langle \alpha^{n-1}\rangle \smile \alpha - \volume(\alpha) = 0,\end{align*}
which, combined with (\ref{ine-RkchanDk}), implies 
$$\lim_{k\to \infty}\int_{\supp (D_k)} \omega_k^{n-1}\wedge R_k=0.$$
This coupled with (\ref{ine-uocluongEktichphan}) and (\ref{eq-limksuppD_k}) gives the desired assertion. 
\end{proof}

\bibliography{biblio_family_MA,biblio_Viet_papers,bib-kahlerRicci-flow}
\bibliographystyle{alpha}

\bigskip

\noindent
\Addresses
\end{document}